\def\={=\!\!}
\newcommand*\dep{{=\mkern-1.2mu}}

\documentclass[12pt]{article}
\usepackage{amsthm} 
\usepackage{graphicx} 
\usepackage{amssymb,latexsym,times}
\newcommand{\psfrag}[2]{}
\usepackage{makeidx,proof}
\usepackage{doi}
\usepackage{enumerate}
\newtheorem{theorem}{Theorem}
\newtheorem{definition}[theorem]{Definition}
\newtheorem{lemma}[theorem]{Lemma}

\newcommand{\rind}[3]{{#2}~\bot_{#1}~{#3}}
\title{On Dependence Logic
}
\author{Pietro Galliani\thanks{Research partially supported by the EUROCORES LogICCC LINT programme, by the V\"ais\"al\"a Foundation and by by Grant 264917 of the Academy of Finland.}
\\ Department of Mathematics and Statistics\\ University of Helsinki, Finland\\ \and Jouko V\"a\"an\"anen\thanks{Research partially supported by
grant 40734 of the Academy of Finland and the EUROCORES LogICCC LINT programme.}\\ Department of Mathematics and Statistics\\ University of Helsinki, Finland\\ and \\
Institute for Logic, Language and Computation \\ University of Amsterdam, The Netherlands}
\date{}
\begin{document}
\maketitle
\def\vx{\vec{x}}
\def\vxp{\vec{x}\hspace{1pt}'}
\def\vyp{\vec{y}\hspace{1pt}'}
\def\vxpp{\vec{x}\hspace{1pt}''}
\def\vypp{\vec{y}\hspace{1pt}''}
\def\vzp{\vec{z}\hspace{1pt}'}
\def\vzpp{\vec{z}\hspace{1pt}''}
\def\vy{\vec{y}}
\def\vz{\vec{z}}
\def\vu{\vec{u}}
\def\vv{\vec{v}}
\def\vw{\vec{w}}
\def\boto{\ \bot\ }
\def\botop{\ \bot'\ }
\newcommand{\idep}[3]{#1\ \bot_{#2}\ #3}
\newcommand{\idepb}[2]{#1\ \bot\ #2}
\def\yxz{\idep{\vy}{\vx}{\vz}}
\def\xzy{\idep{\vx}{\vz}{\vy}}
\def\xzx{\idep{\vx}{\vz}{\vx}}
\def\xyx{\idep{\vx}{\vy}{\vx}}
\def\xzu{\idep{\vx}{\vz}{\vu}}
\def\uzy{\idep{\vu}{\vz}{\vy}}
\def\yzy{\idep{\vy}{\vz}{\vy}}
\def\yxy{\idep{\vy}{\vx}{\vy}}
\def\xyu{\idep{\vx}{\vy}{\vu}}
\def\xy{\idep{\vx}{\vy}}
\def\dyx{\dep(\vy,\vx)}


\def\mm{\mathfrak{M}}

\section{Introduction}

 The goal of dependence logic is to establish a basic theory of dependence and independence underlying such seemingly unrelated subjects  as causality, random variables, bound variables in logic, database theory, the theory of social choice, and even  quantum physics. There is an avalanche of new results in this field demonstrating remarkable convergence. The concepts of (in)dependence in the different fields of humanities and sciences have surprisingly much in common and a common logic is starting to emerge.

Dependence logic \cite{MR2351449} arose from the compositional semantics of Wilfrid Hodges  \cite{MR98k:03068} for the  independence friendly logic \cite{MR90j:03054,MR2807973}. In dependence logic the basic semantic concept is {\em not} that of an assignment $s$ satisfying a formula $\phi$ in a model $\mm$,
$$\mm\models_s\phi,$$ as in first order logic, but rather the concept of a {\em set} $S$ of assignments satisfying $\phi$ in $\mm$,
$$\mm\models_S\phi.$$ Defining satisfaction relative to a {\em set} of assignments opens up the possibility to express dependence phenomena, roughly as passing in propositional logic from one valuation to a Kripke model leads to the possibility to express modality. The focus in dependence logic is not on truth values but on variable values. We are interested in dependencies between individuals rather than between propositions.

In \cite{MR1433259} Johan van Benthem 
 writes: 

\begin{equation}\label{vB}
\mbox{\small\it \begin{tabular}{l}
``Sets of assignments $S$ encode several kinds of `dependence'\\
 between variables. There may not be one single intuition. \\
 `Dependence' may mean functional dependence \\
  (if two assignments agree in $S$ on $x$, they also agree on $y$),\\
but also other kinds of `correlation' among value ranges. \\
... \\
Different dependence relations may have different mathematical\\
 properties
 and suggest different logical formalisms.''  
\end{tabular}
}
\end{equation}
This is actually how things have turned out. For a start,
using the concept of functional dependence it is  possible, as Wilfrid Hodges \cite{MR98k:03068}  demonstrated, to define compositionally\footnote{Before \cite{MR98k:03068} it was an open question whether a compositional semantics can be given to independence friendly logic. } the semantics of independence friendly logic, the extension of first order logic by the quantifier
$$\exists x/y\phi\hspace{5mm}\mbox{  i.e. ``there is an $x$, independently of $y$, such that $\phi$''},$$ as follows: Suppose $S$ is a team of assignments, a ``plural state",  in a model $\mm$. Then
 $$\mm\models_S\exists x/y\phi$$
if and only if there is another set $S'$ such that $$\mm\models_{S'}\phi$$
and the following ``transition"-conditions hold:
\begin{itemize}

\item If $s\in S$, then there is $s'\in S'$ such that if $z$ is a variable other than $x$, then $s(z)=s'(z)$.
\item If $s'\in S'$, then there is $s\in S$ such that if $z$ is a variable other than $x$, then $s(z)=s'(z)$.

\item If $s,s'\in S'$ and $s(z)=s'(z)$ for all variables other than $y$ or $x$, then $s(x)=s'(x)$.

\end{itemize}  
In a sense, independence friendly logic is a logical formalism suggested by the functional dependence relation, but its origin is in game theoretical semantics, not in dependence relations. With dependence logic the situation is different. It was directly inspired by the functional dependence relation introduced by Wilfrid Hodges. 

Peter van Emde Boas pointed out to the second author in the fall of 2005 that the functional dependence behind dependence logic is known in database theory \cite{armstrong}. This led the second author to realize---eventually---that the dependence we are talking about here is not just about variables in logic but a much more general phenomenon, covering such diverse areas as algebra, statistics, computer science, medicine, biology, social science, etc. 

  
As Johan van Benthem points out in (\ref{vB}), there are different dependence intuitions. Of course the same is true of intuitions about independence. For some time it was not clear what would be the most natural concept of {\em independence}. There was the obvious but rather weak form of independence of $x$ from $y$ as dependence of $x$ on some  variable $z$ other than $y$. Eventually a strong form of independence was introduced in \cite{GV}, which has led to a breakthrough in our understanding of dependence relations and their role.     

We give an overview of some developments in dependence logic (Section~\ref{dl}) and independence logic (Section~\ref{il}). This is a tiny selection, intended for a newcomer, from a rapidly growing literature on the topic. Furthermore, in Section \ref{cil} we discuss conditional independence atoms and we prove a novel result -- that is, that conditional and non-conditional independence logic are equivalent. Finally, in Section \ref{bel} we briefly discuss an application of our logics to belief representation.
\section{Functional dependence}\label{dl}

The approach of \cite{MR2351449} is that one should look for the strongest concept of dependence and use it to define weaker versions. Conceivably one could do the opposite, start from the weakest and use it to define stronger and strong concepts. The weakest dependence concept---whatever it is---did not offer itself immediately, so the strongest was more natural to start with. The wisdom of focusing in the extremes lies in the hope that the extremes are most likely to manifest simplicity and robustness, which would make them susceptible to a theoretical study.

Let us start with the strongest form of dependence, functional dependence. We use the vector notation $\vx$ for finite sequences $x_1,\ldots,x_n$ of variables\footnote{Or attributes, something that has a value.}. We add to first order logic\footnote{The basic ideas can be applied to almost any logic, especially to modal logic.}
new atomic formulas%
\begin{equation}\label{dep1}
\dep(\vy,\vx),
\end{equation}
with the intuitive meaning $$\mbox{``the  $\vy$ totally determine the
 $\vx$}.$$ In other words, the meaning of (\ref{dep1}) is that the values of the variables $\vy$ functionally determine the values of the variables $\vx$. We think of the atomic formulas (\ref{dep1}) on a par with the atomic formula $x=y$. In particular, the idea is that the formula (\ref{dep1}) is a purely logical expression, not involving any non-logical symbols, in particular no function symbol for the purported function manifesting the functional dependence.  
 
The best way to understand the concept (\ref{dep1}) is to give it exact semantics: 
To this end, suppose $\mm$ is a model. Suppose $S$ is a set of assignments into $M$ (or a {\em team} as such sets are called). We define:

\begin{definition}\label{def2}
The team  $S$ satisfies $\dep(\vy,\vx)$ in $\mm$, in symbols 
 $$\mm\models_S \ \dep(\vy,\vx) $$
 if \begin{equation}\label{funct1}\forall s,s'\in S(s(\vy)=s'(\vy)\to s(\vx)=s'(\vx)).\end{equation}
\end{definition}

One may ask, why not define the meaning of $\dep(y,x)$ as ``there is a function which maps $y$ to $x$"? The answer is that if we look at the meaning of $\dep(y,x)$ under {\em one} assignment $s$, then there {\em always} is a function $f$ mapping $s(y)$ to $s(x)$, namely the function $\{(s(y),s(x))\} $, and if we look at the meaning of $\dep(y,x)$ under {\em many} assignments, a team, then (\ref{funct1}) is indeed equivalent to the statement that there is a function mapping $s(y)$ to $s(x)$ for all $s$ in the team.

A special case of $\dep(\vy,\vx)$ is $\dep(\vx)$, the {\em constancy atom}. The intuitive meaning of this atom is that the value of $\vx$ is constant in the team. It results from $\dep(\vy,\vx)$ when $\vy$ is the empty sequence.

%
 
%

%
 
Functional dependence has been studied in database theory and some basic properties, called {\bf Armstrong's Axioms} have been isolated \cite{armstrong}. These axioms state the following properties of $\dep(\vy,\vx)$:

\begin{description}
\item[(A1)] $\dep(\vx,\vx)$. Anything is functionally dependent of itself.
\item[(A2)] If $\dep(\vy,\vx)$ and $\vy\subseteq\vz$, then $\dep(\vz,\vx)$. Functional dependence is preserved by increasing input data.
\item[(A3)] If $\vy$ is a permutation of $\vz$, $\vu$ is a permutation of $\vx$, and $\dep(\vz,\vx)$, then $\dep(\vy,\vu)$. Functional dependence does not look at the order of the variables.
\item[(A4)] If $\dep(\vy,\vz)$  and $\dep(\vz,\vx)$, then  $\dep(\vy,\vx)$. Functional dependences can be transitively composed. 
\end{description} 

The following result is well-known in the database community and included in textbooks of database theory:\footnote{See e.g. \cite{mannila}.}
\begin{theorem}[\cite{armstrong}]The axioms (A1)-(A4) are complete in the  sense that a relation $\dep(\vec{x},\vec{y})$ follows by the rules (A1)-(A4) from a set $\Sigma$ of relations of the same form if and only if every team which satisfies $\Sigma$ satisfies $\dep(\vec{y},\vec{x})$.
\end{theorem}

\begin{proof}
Suppose $\dep(\vy,\vx)$ does {not} follow by the rules from  a set $\Sigma$ of atoms. Let $V$ be the set of variables $z$ such that $\dep(\vy,z)$ follows by the rules from $\Sigma$. Let $W$ be the remaining variables in $\Sigma\cup\{\dep(\vy,\vx)\}$. Thus $\vx\cap W\ne\emptyset$. Consider the model $\{0,1\}$ of the empty vocabulary and the team
\medskip

\begin{center}
\begin{tabular}{|c|c|c|c|c|c|c|c|}
\hline
\multicolumn{4}{|c}{The variables in $V$}&
\multicolumn{4}{|c|}{The variables in $W$}\\
\hline
0&0&$\ldots$&0&0&$\ldots$&$\ldots$&0\\
0&0&$\ldots$&0&1&1&$\ldots$&1\\
\hline
\end{tabular}
\end{center}
\medskip

\noindent The atom $\dep(\vy,\vx)$ is not true in this team, because $\vy\subseteq V$ and $\vx\cap W\ne\emptyset$. Suppose then $\dep(\vv,\vw)$ is one of the assumptions. If each $v$ is in $V$, then so is each $w$ so they all get value $0$. On the other hand, if some $v$ is in $W$, it gets in this team two values, so it cannot violate dependence.

\end{proof}

We now extend the truth definition (\ref{def2}) to the full first order logic augmented by the dependence atoms $\dep(\vx,\vy)$. To this end, let $s(a/x)$ denote the assignment which agrees with $s$ except that it gives $x$ the value $a$. We define for formulas which have negation in front of atomic formulas only: 

\medskip

{\small\begin{equation}\label{truth}
\left.\mbox{\begin{tabular}{lcl}

$\mm\models_S x=y$&$\iff$&$\forall s\in S(s(x)=s(y))$.\\
$\mm\models_S \neg x=y$&$\iff$&$\forall s\in S(s(x)\ne s(y))$.\\

 $\mm\models_S R(x_1,\ldots,x_n)$&$\iff$&$\forall s\in S((s(x_1),\ldots,s(x_n))\in R^\mm)$.\\

$\mm\models_S \neg R(x_1,\ldots,x_n)$&$\iff$&$\forall s\in S((s(x_1),\ldots,s(x_n))\notin R^\mm)$.\\

$\mm\models_S\phi\wedge\psi$&$\iff$&$\mm\models_S\phi\mbox{ and }\mm\models_S\psi$.\\

$\mm\models_S\phi\vee\psi$&$\iff$&$\mbox{There are $S_1$ and $S_2$ such that }$\\
&&\mbox{$S=S_1\cup S_2$, }$ \mm\models_{S_1}\phi, \mbox{ and }\mm\models_{S_2}\psi$.\\

$\mm\models_S\exists x\phi$&$\iff$&$\mm\models_{S'}\phi\mbox{ for some $S'$ such that  }$\\

&&$\forall s\in S\ \exists a\in M(s(a/x)\in S')$\\

$\mm\models_S\forall x\phi$&$\iff$&$\mm\models_{S'}\phi\mbox{ for some $S'$ such that  }$\\

&&$\forall s\in S\ \forall a\in M(s(a/x)\in S')$\\

\end{tabular}}\right\}
\end{equation}
}
\medskip

It is easy to see that for formulas not containing any dependence atoms, that is, for pure first order formulas $\phi$, 
$$\mm\models_{\{s\}}\phi\iff \mm\models_s\phi$$
and 
$$\mm\models_S\phi\iff\forall s\in S(\mm\models_s\phi),$$
where $\mm\models_s\phi$ has its usual meaning. This shows that the truth conditions  
(\ref{truth}) agree with the usual Tarski truth conditions for first order formulas. Thus considering the ``plural state'' $S$ rather than individual ``states'' $s$ makes no difference for first order logic, but it makes it possible to give the dependence atoms $\dep(\vx,\vy)$ their intended meaning.





What about  axioms for non-atomic formulas of dependence logic? Should we adopt new axioms, apart from the Armstrong Axioms [A1-A4]? There is a problem! Consider the sentence

\begin{equation}
\label{infty}
\exists x\forall y\exists z(\dep(z,y)\wedge \neg z=x).
\end{equation}
It is easy to see that this sentence is satisfied by a team in a model $\mm$ if and only $M$ is infinite. As a result, by general considerations going back to G\"odel's Incompleteness Theorem, the semantic consequence relation
$$\phi\models\psi\iff\forall\mm\forall S(\mm\models_S\phi\to\mm\models_S\psi)$$ is non-arithmetical. Thus there cannot be any completeness theorem in the usual sense. However, this does not prevent us from trying to find axioms and rules which are as complete as possible.
This is what is done in  \cite{kontvaa}, where a complete axiomatization is given for {\em first order} consequences of dependence logic sentences. The axioms are a little weaker than standard first order axioms when applied to dependence formulas, but on the other hand there are two special axioms for the purpose of  dealing with dependence atoms as parts of formulas in a deduction. Rather than giving all details (which can be found in \cite{kontvaa}) we give just an example of the use of both new rules.

Suppose we are given $\epsilon,x,y$ and $f$, and we have already concluded, in the middle of some argument, the following:
\bigskip

{\it 

\noindent\hspace{2cm}\begin{tabular}{l}
If $\epsilon>0$, then  there is $\delta>0$
 depending only on $\epsilon$ such that\\
   if $|x-y|<\delta$, then $|f(x)-f(y)|<\epsilon$.
\end{tabular}
}
\bigskip

\noindent  By merely logical reasons we should be able to conclude
\bigskip

{\it
\noindent\hspace{2cm}\begin{tabular}{l}
There is $\delta>0$ depending only on $\epsilon$ such that\\
  {if $\epsilon>0$} and $|x-y|<\delta$,
   then $|f(x)-f(y)|<\epsilon$.
\end{tabular}
}

\bigskip

\noindent Note that ``{\it depending only on $\epsilon$}" has moved from inside the implication to outside of it. The new rule of dependence logic, isolated in \cite{kontvaa}, which permits this, is called {\em {Dependence Distribution Rule}}.
Neither first order rules nor Armstrong's Axioms  give this because neither of them gives any clue of how to deal with dependence atoms as parts of bigger formulas.

Here is another example of inference in dependence logic: Suppose we have arrived at the following formula in the middle of some argument:

\bigskip

{\it 
\noindent\hspace{2cm}\begin{tabular}{l}
For every $x$ and every $\epsilon>0$
there is $\delta>0$
 depending only on $\epsilon$\\ such that for all $y$, 
 if $|x-y|<\delta$, then $|f(x)-f(y)|<\epsilon$.
\end{tabular}
\bigskip
}

\noindent  On merely logical grounds we should be able to make the following conclusion:
\bigskip
{\it 

\noindent\hspace{2cm}\begin{tabular}{l}
For every $x$ and every $\epsilon>0$ there is $\delta>0$\\ 
such that for all $y$, if $|x-y|<\delta$, then $|f(x)-f(y)|<\epsilon$,\\
and moreover, for any other $x'$ and $\epsilon'>0$ there is $\delta'>0$\\
 such that for all $y'$, if $|x'-y'|<\delta'$, then $|f(x')-f(y')|<\epsilon$\\
 and if $\epsilon=\epsilon'$, then $\delta=\delta'$.
\end{tabular}
}

\bigskip
\noindent The new rule, isolated in \cite{kontvaa} which permits this step is called 
{\em Dependence Elimination Rule}, because the dependence atom {\it ``depending only on $\epsilon$"} has been entirely eliminated. The conclusion is actually first order, that is, without any occurrence of dependence atoms.

The first author \cite{pietro1} has given an alternative complete axiomatization, not for first order consequences of dependence sentences, but for dependence logic consequences of first order sentences. Clearly, more results about partial axiomatizations of the logical consequence relation in dependence logic can be expected in the near future. 

An important property of dependence logic is the {\em downward closure} \cite{MR99i:03044}: If $\mm\models_S\phi$ and $S'\subseteq S$, then $\mm\models_{S'}\phi$. It is a trivial matter to prove this by induction on the length of the formula. Once the downward closure is established it is obvious that we are far from having a negation in the sense of classical logic. Intuitively, dependence is a restriction of freedom (of values of variables in assignments). When the team gets smaller there is even less freedom. This intuition about the nature of dependence prevails in all the logical operations of dependence logic. Since dependence formulas are easily seen to be representable in existential second order logic, the following result shows that downward closure is really {\em the} essential feature of dependence logic:

\begin{theorem}[\cite{kont}]\label{main1} Let us fix a vocabulary $L$ and an $n$-ary predicate symbol $S\notin L$. Then:
\begin{itemize}
\item For every $L$-formula $\phi(x_1,...,x_n)$ of dependence logic  there is an existential second order $L\cup\{S\}$-sentence $\Phi(S)$, closed downward with respect to $S$, such that for all $L$-structures $M$  and all teams $X$:
\begin{equation}
\mm\models_X\phi(x_1,...,x_n)\iff \mm\models\Phi(X).
\label{eq}
\end{equation}

\item For every existential second order $L\cup\{S\}$-sentence $\Phi(S)$, closed downward with respect to $S$,  there exists an $L$-formula $\phi(x_1,...,x_n)$ of dependence logic  such that (\ref{eq}) holds  for all $L$-structures $M$ and all teams $X\ne\emptyset$. 
\end{itemize} \end{theorem}

This shows that dependence logic is maximal with respect to the properties of being expressible in existential second order logic and being downward closed. This theorem is also the source of the main model theoretical properties of dependence logic. The Downward L\"owenheim-Skolem Theorem, the Compactness Theorem  and the Interpolation Theorem are immediate corollaries. Also, when the above theorem is combined with the Interpolation Theorem of first order logic, we get the  fact that dependence logic sentences $\phi$ for which there exists a dependence logic sentence $\psi$ such that for all $\mm$ $$\mm\models\psi\iff\mm\not\models\phi$$ are first order definable. So not only does dependence logic not have the classical negation, the only sentences that have a classical negation are the first order sentences.


\section{Independence logic}\label{il}

Independence logic was introduced in \cite{GV}. Before going into the details, let us look at the following precedent:

In \cite{MR1433259} Johan van Benthem suggested, as an example of an ``other kind of correlation'' than functional dependence, the following dependence relation for a team $S$ in a model $\mm$:
 \begin{equation}
\label{indepp}\exists a\in M\exists b\in M(\{s(x):s\in S, s(y)=a\}\ne\{s(x):s\in S, s(y)=b\}).
\end{equation} 
The opposite of this would be
\begin{equation}
\label{indeppp}\forall a\in M\forall b\in M(\{s(x):s\in S, s(y)=a\}=\{s(x):s\in S, s(y)=b\}),
\end{equation} 
which is  a kind of independence of $x$ from $y$, for if we take $s\in S$ and we are told what $s(y)$ is, we have learnt nothing about $s(x)$, because for each $a\in M$ the set $$\{s(x):s\in S, s(y)=a\}$$ is the same. This is the idea behind the independence atom $\vx\boto\vy$: the values of $\vx$ should not reveal anything about the values of $\vy$ and vice versa. More exactly, suppose $\mm$ is a model and $S$ is a team of assignments into $M$. We define:

\begin{definition}\label{depdef}
A team $S$ satisfies the atomic formula $\vx\boto \vy$ in $\mm$ if \begin{equation}
\label{indep}\forall s,s'\in S\exists s''\in S(s''(\vy)=s(\vy)\wedge s''(\vx)=s'(\vx)).
\end{equation} 
\end{definition}


We can immediately observe that a constant variable is independent of every  variable, including itself. To see this, suppose $x$ is constant in $S$. Let $y$ be any variable, possibly $y=x$. If $s,s'\in S$ are given, we need $s''\in S$ such that $s''(x)=s(x)$ and $s''(y)=s'(y)$. We can simply take $s''=s'$. Now $s''(x)=s(x)$, because $x$ is constant in $S$. Of course, $s''(y)=s'(y)$. Conversely, if $x$ is independent of every  variable, it is clearly constant, for it would have to be independent of itself, too. So we have $$\dep(\vx)\iff \vx\boto \vx.$$

We can also immediately observe the symmetry of independence, because  criterion (\ref{indep}) is symmetrical in $x$ and $y$. More exactly, $s''(y)=s(y)\wedge s''(x)=s'(x)$ and  $s''(x)=s'(x)\wedge s''(y)=s(y)$ are trivially equivalent.



Dependence atoms were governed by Armstrong's Axioms. Independence atoms have their own axioms introduced in the context of random variables in \cite{MR1097266}:

\begin{definition} The following rules are  the {\bf Independence Axioms}
\begin{enumerate}
\item $\vx\boto\emptyset$ (Empty Set Rule)

\item If $\vx\boto \vy$, then $\vy\boto \vx$ (Symmetry Rule).
\item If $\vx\boto \vy\vz$, then $\vx\boto \vy$ (Weakening Rule).
\item If $\vx\boto \vx$, then $\vx\boto \vy$ (Constancy Rule).
\item If $\vx\boto \vy$ and $\vx\vy\boto\vz$, then $\vx\boto \vy\vz$ (Exchange Rule).

\end{enumerate} 
\end{definition}

Note that $xy\boto xy$ is derivable from $x\boto x$ and $y\boto y$, by means of the Empty Set Rule, the Constancy Rule and the Exchange Rule.
\medskip



It may seem that independence must have much more content than what these four axioms express, but they are actually complete in the following sense\footnote{This was originally proved for random variables in \cite{MR1097266} and then adapted for databases in \cite{KLV}.}: 

\begin{theorem}[Completeness of the Independence Axioms, \cite{MR1097266}]
If $T$ is a finite set of independence atoms of the form $\vu\boto \vv$ for various $\vu$ and $\vv$, then $\vy\boto \vx$ follows from $T$ according to the above rules if and only if every team that satisfies $T$ also satisfies $\vy\boto \vx$. 

\end{theorem}

\begin{proof} We adapt the proof of \cite{MR1097266} into our framework.
Suppose $\vx\boto\vy$ follows semantically from $\Sigma$ but does not follow  by the above rules. W.l.o.g. $\Sigma$ is closed under the rules. We may assume that $\vx$ and $\vy$ are minimal, that is, if $\vec{x}\hspace{1pt}'\subseteq \vx$ and $\vec{y}\hspace{1pt}'\subseteq \vy$ and at least one containment is proper, then if $\vx\hspace{1pt}'\boto\vy\hspace{1pt}'$ follows from $\Sigma$ semantically, it also follows by the rules. It is easy to see that if $\Sigma\models u\boto u$, then 
$\Sigma\vdash u\boto u$.

Suppose $\vx=(x_1,\ldots,x_l)$ and $\vy=(y_1,\ldots,y_m)$. Let $\vz=(z_1,\ldots,z_k)$ be the remaining variables. Wlog, $l\ge 1$ and $m\ge 1$, $x_1\boto x_1\notin\Sigma$, and $x_1\notin\{y_1,\ldots,y_m\}$.

We construct a team $S$ in a 2-element model $M=\{0,1\}$ of the empty vocabulary as follows: We take  to $S$ every $s:\vx\vy\vz\to M$, which satisfies $s(u)=0$ for $u$ such that $u\boto u\in \Sigma$ and in addition
$$s(x_1)=\mbox{ the number of ones in $s[\{x_2,\ldots,x_l,y_1,\ldots,y_m\}]$ mod 2}$$

\medskip

\noindent Claim 1: $\vx\boto\vy$ is not true in $S$. Suppose otherwise. Consider the following two assignments in $S$:

\medskip

\begin{center}
\begin{tabular}{l|ccccc|}
\hline
&$x_1$&other $x_i$&$y_1$&other $y_i$&other\\
\hline
$s$&1&0&1&0&0\\
$s'$&0&0&0&0&0\\
\hline\end{tabular}
\end{center}
If $s''$ is such that $s''(\vx)=s(\vx)$ and $s''(\vy)=s'(\vy)$, then $s''\notin S$. Claim 1 is proved.

\medskip
\noindent Claim 2: $S$ satisfies all the independence atoms  in $\Sigma$.
Suppose $\vv\boto\vw\in S$. If  either $\vv$ or $\vw$ contains only variables in $Z$, then the claim is trivial, as then either $\vv$ or $\vw$ has in $S$ all possible binary sequences.
So let us assume that 
both $\vv$ and $\vw$ meet $\vx\vy$.
If  $\vv\vw$ does not cover all of $\vx\vy$, then  $S$ satisfies $\vv\boto\vw$, because we can fix parity on the variable in $\vx\vy$ which does not occur in $\vv\vw$. So let us assume  $\vv\vw$ covers all of $\vx\vy$. Thus $\vv=\vxp\vyp\vzp$ and $\vw=\vxpp\vypp\vzpp$, where $\vxp\vxpp=\vx$ and $\vyp\vypp=\vy$. 
W.l.o.g., $\vxp\ne\emptyset$ and $\vxp\vyp\ne\vx\vy$. 
By minimality $\vxp\boto\vyp\in \Sigma$ and $\vxpp\boto\vy\in \Sigma$. Since $\vv\boto\vw\in \Sigma$, a couple of applications of the Exchange and Weakening Rules gives $\vxp\vyp\boto\vxpp\vypp\in\Sigma$. But then $\vxp\vxpp\boto\vyp\vypp\in \Sigma$, contrary to assumption.

\end{proof}

We can use the conditions (\ref{truth}) to extend the truth definition to the entire {\em independence logic}, i.e. the extension of first order logic by the independence atoms. Can we axiomatize logical consequence in independence logic?
The answer is again no, and for the same reason as for dependence logic:
Recall that the sentence (\ref{infty}) characterizes infinity and ruins any hope to have a completeness theorem for dependence logic. We can do the same using independence atoms:

\begin{lemma}
The sentence
\begin{equation}
\label{infty1}
\exists z\forall x\exists y\forall u\exists v(xy\boto uv\wedge(x=u\leftrightarrow y=v)\wedge \neg v=z)
\end{equation}
is true exactly in infinite models.
\end{lemma}

The conclusion is that the kind of dependence relation needed for expressing infinity can be realized either by the functional dependence relation or by the independence relation. Another such example is parity in finite models. The following two sentences, the first one with a dependence atom and the second with an independence atom, both express the evenness of the size of a finite model:

$$\forall x\exists y\forall u\exists v(\dep(u,v)\wedge (x=v\leftrightarrow y=u)\wedge \neg x=y)$$
$$\forall x\exists y\forall u\exists v(xy\boto uv\wedge (x=v\leftrightarrow y=u)\wedge \neg x=y)$$

The fact that we could express, at will,  both infinity and evenness by means of either dependence atoms or independence atoms,  is not an accident. Dependence logic and independence logic have overall the same expressive power:

\begin{theorem}
The following are equivalent:
\begin{description}

\item[(1)] $K$ is definable by a sentence of the extension of first order logic by the dependence atoms.
\item[(2)] $K$ is definable by a sentence of the extension of first order logic by the independence atoms.
\item[(3)] $K$ is definable in existential second order logic.

\end{description}
\end{theorem}

\begin{proof}
The equivalence of (1) and (3), a consequence of results in \cite{MR44:1546}  and \cite{MR43:4646}, as observed in \cite{MR99i:03044}, is proved in \cite{MR2351449}.
So it suffices to show that (1) implies (2). We give only the main idea. Sentences referred to in (1) have a normal form \cite{MR2351449}. Here is an example of a sentence in such a normal form
$$
\forall x\forall y\exists v\exists w(\dep(x,v)\wedge\dep(y,w)\wedge\phi(x,y,v,w)),
$$
where $\phi(x,y,v,w)$ is a quantifier free first order formula. This sentence can be expressed in terms of independence atoms as follows:
$$
\forall x\forall y\exists v\exists w(xv\bot y\wedge yw\bot xv\wedge\phi(x,y,v,w)).
$$

\end{proof}

Note that independence, as we have defined it, is not the negation of dependence. It is rather a very strong denial of dependence. However, there are uses of the concepts of dependence and independence where the negation of dependence {\em is} the same as independence. An example is vector spaces.

There is an earlier common use of the concept of independence in logic, namely the independence of a set $\Sigma$ of axioms from each other. This is usually taken to mean that no axiom is provable from the remaining ones. By G\"odel's Completeness Theorem this means the same as having for each axiom $\phi\in \Sigma$ a model of the remaining ones $\Sigma\setminus\{\phi\}$ in which $\phi$ is false. This is not so far from the independence concept $\vy\boto\vx$. Again, the idea is that from the truth of $\Sigma\setminus\{\phi\}$ we can say nothing about the truth-value of $\phi$. This is the sense in which Continuum Hypothesis (CH) is independent of ZFC. Knowing the ZFC axioms gives as no clue as to  the truth or falsity of CH. In a sense, our independence atom $\vy\boto\vx$ is the familiar concept of independence transferred from the world of formulas to the world of elements of models, from truth values to variable values. 

\section{Conditional Independence}\label{cil}
The independence  atom $\vy\boto \vx$ turns out to be a special case of the more general atom $\yxz$, the intuitive meaning of which is that the variables $\vy$ are totally independent of the variables $\vz$ when the variables $\vx$ are kept fixed (see \cite{GV}). Formally, 
\begin{definition}
A team $S$ satisfies the atomic formula $\yxz$ in $\mm$ if 
\[
\forall s,s'\in S ( s(\vx) = s'(\vx) \rightarrow \exists s''\in S(s''(\vx\vy)=s(\vx\vy)\wedge s''(\vz)=s'(\vz))).
\]
\end{definition}

Some of the rules that this ``conditional'' independence notion obeys are

\begin{description}
\item[Reflexivity:] $\rind{\vx}{\vx}{\vy}$;
\item[Symmetry:] If $\rind{\vx}{\vy}{\vz}$, then $\rind{\vx}{\vz}{\vy}$;
\item[Weakening:] If $\rind{\vx}{\vy y'}{\vz z'}$, then $\rind{\vx}{\vy}{\vz}$;
\item[First Transitivity:] If $\rind{\vz}{\vx}{\vy}$ and $\rind{\vz\vx}{\vu}{\vy}$, then $\rind{\vz}{\vu}{\vy}$;
\item[Second Transitivity:] If $\rind{\vz}{\vy}{\vy}$ and $\rind{\vy}{\vz\vx}{\vu}$, then $\rind{\vz}{\vx}{\vu}$.
\item[Exchange:] If $\rind{\vz}{\vx}{\vy}$ and $\rind{\vz}{\vx\vy}{\vu}$, then $\rind{\vz}{\vx}{\vy\vu}$.
\end{description} 
Are these axioms complete? More in general, is it possible to find a finite, decidable axiomatization for the consequence relation between conditional independence atoms? 

The answer is negative. Indeed, in \cite{herrmann95,herrmann06} Hermann proved that the consequence relation between conditional independence atoms is undecidable; and as proved by Parker and Parsaye-Ghomi in \cite{parker80}, it is not possible to find a finite and complete axiomatization for these atoms. However, the consequence relation is recursively enumerable, and in \cite{naumovre} Naumov and Nicholls developed a proof system for it.

The logic obtained by adding conditional independence atoms to first order logic will be called in this paper \emph{conditional independence logic}. It is clear that it contains (nonconditional) independence logic; and furthermore, as discussed in \cite{GV}, it also contains dependence logic, since a dependence atom $\dep(\vx, \vy)$ can be seen to be equivalent to $\rind{\vx}{\vy}{\vy}$. It is also easy to see that every conditional independence logic sentence is equivalent to some $\Sigma_1^1$ sentence, and therefore that conditional independence logic is equivalent to independence logic and dependence logic with respect to sentences. 

But this leaves open the question of whether every conditional independence logic formula is equivalent to some independence logic one. In what follows, building on the analysis of the expressive power of conditional independence logic of \cite{galliani12},\footnote{In that paper, conditional independence logic is simply called ``independence logic''. After all, the two logics are equivalent.} we prove that independence logic and conditional independence logic are indeed equivalent. 

In order to give our equivalence proof we first need to mention two other atoms, the inclusion atom $\vx \subseteq \vy$ and the exclusion atom $\vx ~|~ \vy$. These atoms correspond to the database-theoretic inclusion \cite{fagin81,casanova82} and exclusion \cite{casanova83} dependencies, and hold in a team if and only if no possible value for $\vx$ is also a possible value for $\vy$ and if every possible value for $\vx$ is a possible value for $\vy$ respectively. More formally, 
\begin{definition}
A team $S$ satisfies the atomic formula $\vx \subseteq \vy$ in $\mm$ if 
\[
\forall s\in S \exists s'\in S(s'(\vy)=s(\vx))
\]
and it satisfies the atomic formula $\vx ~|~ \vy$ in $\mm$ if 
\[
\forall s, s' \in S(s(\vx) \not = s'(\vy)).
\]
\end{definition}

As proved in \cite{galliani12}, 
\begin{enumerate}
\item Exclusion logic (that is, first order logic plus exclusion atoms) is equivalent to dependence logic;
\item Inclusion logic (that is, first order logic plus inclusion atoms) is not comparable with dependence logic, but is contained in (nonconditional) independence logic;
\item Inclusion/exclusion logic (that is, first-order logic plus inclusion and exclusion atoms) is equivalent to \emph{conditional} independence logic (that is, first-order logic plus conditional independence atoms $\rind{\vx}{\vy}{\vz}$).
\end{enumerate}
Thus, if we can show that exclusion atoms can be defined in terms of (nonconditional) independence atoms and of inclusion atoms, we can obtain at once that independence logic contains conditional independence logic (and, therefore, is equivalent to it). But this is not difficult: indeed, the exclusion atom $\vx ~|~ \vy$ is equivalent to the expression
\[
\exists \vz (\vx \subseteq \vz \wedge \vy \boto \vz \wedge \vy \not = \vz).
\]
This can be verified by checking the satisfaction conditions of this formula. But more informally speaking, the reason why this expression is equivalent to $\vx ~|~ \vy$ is that it states that that every possible value of $\vx$ is also a possible value for $\vz$, that $\vy$ and $\vz$ are independent (and therefore, any possible value of $\vy$ must occur together with any possible value of $\vz$), and that $\vy$ is always different from $\vz$. Such a $\vz$ may exist if and only if no possible value of $\vx$ is also a possible value of $\vy$, that is, if and only if $\vx ~|~ \vy$ holds.

Hence we may conclude at once that 
\begin{theorem}
Every conditional independence logic formula is equivalent to some independence logic formula. 
\end{theorem}

In \cite{galliani12} it was also shown the following analogue of Theorem \ref{main1}: 
\begin{theorem}\label{main_ind} Let us fix a vocabulary $L$ and an $n$-ary predicate symbol $S\notin L$. Then:
\begin{itemize}
\item For every $L$-formula $\phi(x_1,...,x_n)$ of conditional independence logic  there is an existential second order $L\cup\{S\}$-sentence $\Phi(S)$ such that for all $L$-structures $M$  and all teams $X$:
\begin{equation}
\label{eq2}
\mm\models_X\phi(x_1,...,x_n)\iff \mm\models\Phi(X).
\end{equation}

\item For every existential second order $L\cup\{S\}$-sentence $\Phi(S)$ there exists an $L$-formula $\phi(x_1,...,x_n)$ of conditional independence logic  such that (\ref{eq2}) holds  for all $L$-structures $M$ and all teams $X\ne\emptyset$. 
\end{itemize} \end{theorem}
Due to the equivalence between independence logic and conditional independence logic, the same result holds if we only allow nonconditional independence atoms. In particular, this implies that over finite models independence logic captures precisely the NP properties of teams. 




\section{Belief Representation and Belief Dynamics}\label{bel}
Given a model $\mm$, a variable assignment $s$ admits a natural interpretation as the representation of a possible \emph{state of things}, where, for every variable $v$, the value $s(v)$ corresponds to a specific \emph{fact} concerning the world. To use the example discussed in Chapter 7 of \cite{pietro1}, let the elements of $\mm$ correspond to the participants to a competition: then the values of the variables $x_1$, $x_2$ and $x_3$ in an assignment $s$ may correspond respectively to the first-, second- and the third-placed players. 

With respect to the usual semantics for first order logic, a first order formula represents a \emph{condition} over assignments. For example, the formula 
\[
	\phi(x_1, x_2, x_3) := (\lnot x_1  = x_2) \wedge (\lnot x_2 = x_3) \wedge (\lnot x_1 = x_3)
\]
represents the (very reasonable) assertion according to which the winner, the second-placed player and the third-placed player are all distinct. 

Now, a team $S$, being a set of assignments, represents a set of states of things. Hence, a team may be interpreted as the \emph{belief set} of an agent $\alpha$: $s \in S$ if and only if the agent $\alpha$ believes $s$ to be possible. Moving from assignments to teams, it is possible to associate to each formula $\phi$ and model $\mm$ the family of teams $\{ S : \mm \models_S \phi \}$, and this allows us to interpret formulas as \emph{conditions over belief sets}: in our example, $\mm \models_S \phi(x_1, x_2, x_3)$ if and only if $\mm \models_s \phi(x_1, x_2, x_3)$ for all $s \in S$, that is, if and only if our agent $\alpha$ believes that the winner, the second-placed player and the third-placed player will all be distinct.

However, there is much that first order logic cannot express regarding the beliefs of our agent. For example, there is no way to represent the assertion that the agent $\alpha$ \emph{knows} who the winner of the competition will be: indeed, suppose that a first order formula $\theta$ represents such a property,  and let $s_1$ and $s_2$ be any two assignments with $s_1(x_1) \not = s_2(x_1)$, corresponding to two possible states of things which disagree with respect to the identity of the winner. Then, for $S_1 = \{s_1\}$ and $S_2 = \{s_2\}$, we should have that $\mm \models_{S_1} \theta$ and that $\mm \models_{S_2} \theta$: indeed, both $S_1$ and $S_2$ correspond to belief sets in which the winner is known to $\alpha$ (and is respectively $s_1(x_1)$ or $s_2(x_1)$). But since a team $S$ satisfies a first order formula if and only if all of its assignments satisfy it, this implies that $M \models_{S_1 \cup S_2} \theta$; and this is unacceptable, because if our agent $\alpha$ believes both $s_1$ and $s_2$ to be possible then she does not know whether the winner will be $s_1(x_1)$ or $s_2(x_1)$. 

How to represent this notion of knowledge? The solution, it is easy to see, consists in adding \emph{constancy atoms} to our language: indeed, $\mm \models_S
 =dep(x_1)$ if and only if for any two assignments $s, s' \in S$ we have that $s(x_1) = s'(x_1)$, that is, if and only if all states of things the agent $\alpha$ consider possible agree with respect to the identity of the winner of the competition. 
What if, instead, our agent could infer the identity of the winner from the identity of the second- and third-placed participants? Then we would have that $\mm \models_S \dep(x_2 x_3, x_1)$, since any two states of things which the agent considered possible and which agreed with respect to the identity of the second- and third-placed participants would also agree with respect to the identity of the winner. More in general, a dependence atom $\dep(\vy, \vx)$ describes a form of \emph{conditional knowledge}: $\mm \models_S \dep(\vy, \vx)$ if and only if $S$ corresponds to the belief state of an agent who would be able to deduce the value of $\vx$ from the value of $\vy$. 

On the other hand, independence atoms represent situations of \emph{informational independence}: for example, if $\mm \models_S x_1 \boto x_3$ then, by learning the identity of the third-placed player, our agent could infer nothing at all about the identity of the winner. Indeed, suppose that, according to our agent, it is possible that $A$ will win (that is, there is a $s \in S$ with $s(x_1) = A$) and it is possible that $B$ will place third (that is, there is a $s' \in S$ such that $s'(x_3) = B$). Then, by the satisfaction conditions of the independence atom, there is also a $s'' \in S$ such that $s''(x_1) = A$ and $s''(x_3) = B$: in other words, it is possible that $A$ will be the winner \emph{and} $B$ will place third, and telling our agent that $B$ will indeed place third will not allow her to remove $A$ from her list of possible winners.

Thus, it seems that dependence and independence logic, or at least fragments thereof, may be interpreted as \emph{belief description languages}. This line of investigation is pursued further in \cite{pietro1}: here it will suffice to discuss the interpretation of the \emph{linear implication}\footnote{The name ``linear implication'' is due to the similarity between the satisfaction conditions of this connective and the ones of the implication of linear logic. Another similarity is the following Galois connection: $\theta\models\phi \multimap \psi\iff\theta\vee\phi\models\psi$ \cite{abramskyvaananen}.} $\phi \multimap \psi$, a connective introduced in \cite{abramskyvaananen} whose semantics is given by 
\[
	\mm \models_S \phi \multimap \psi \Leftrightarrow \mbox{ for all } S' \mbox{ such that } \mm \models_{S'} \phi \mbox{ it holds that } \mm \models_{S \cup S'} \psi.
\]
%

How to understand this connective? Suppose that our agent $\alpha$, whose belief state is represented by the team $S$, interacts with another agent $\beta$, whose belief state is represented by the team $S'$: one natural outcome of this interaction may be represented by the team $S \cup S'$, corresponding to the set of all states of things that $\alpha$ \emph{or} $\beta$ consider possible. Then stating that a team $S$ satisfies $\phi \multimap \psi$ corresponds to asserting that whenever our agent $\alpha$ interacts with another agent $\beta$ whose belief state satisfies $\phi$, the result of the interaction will be a belief state satisfying $\psi$: in other words, using the linear implication connective allows us to formulate predictions concerning the future \emph{evolution} of the belief state of our agent.

One can, of course, consider other forms of interactions between agents and further connectives; and quantifiers can also be given natural interpretations in terms of belief updates (the universal quantifier $\forall v$, for example, can be understood in terms of the agent $\alpha$ \emph{doubting} her beliefs about $v$). But what we want to emphasize here, beyond the interpretations of the specific connectives, is that team-based semantics offers a very general and powerful framework for the representation of beliefs and belief updates, and that notions of dependence and independence arise naturally under such an interpretation. This opens up some fascinating -- and, so far, relatively unexplored -- avenues of research, such as for example a more in-depth investigation of the relationship between dependence/independence logic and dynamic epistemic logic (DEL) and other logics of knowledge and belief; and, furthermore, it suggests that epistemic and doxastic ideas may offer some useful inspiration for the formulation and analysis of further notions of dependence and independence.

\section{Concluding remarks}

We hope to have demonstrated that both dependence and independence can be given a logical analysis by moving in semantics from single states $s$ to plural states $S$. Future work will perhaps show that  allowing limited transitions from one plural state to another may lead to decidability results concerning  dependence and independence logic, a suggestion of Johan van Benthem. 

Furthermore, we proved the equivalence between conditional independence logic and independence logic, thus giving a novel contribution to the problem of characterizing the relations between extensions of dependence logic.  

Finally, we discussed how team-based semantics may be understood as a very general framework for the representation of beliefs and belief updates and how notions of dependence and independence may be understood under this interpretation. This suggests the existence of intriguing connections between dependence and independence logic and other formalisms for belief knowledge representation, as well as a possible application for this fascinating family of logics.

\bibliographystyle{plainurl}
\bibliography{indepbib}

\end{document}